\documentclass[11pt]{amsart}
\usepackage{amsmath,amssymb,color,cite,verbatim}
\usepackage[colorlinks,pagebackref,pdftex,bookmarks=false]{hyperref}

\numberwithin{equation}{section}

\newtheorem{thm}[equation]{Theorem}
\newtheorem{prop}[equation]{Proposition}
\newtheorem{lemma}[equation]{Lemma}
\newtheorem{cor}[equation]{Corollary}

\theoremstyle{definition}
\newtheorem{rmk}[equation]{Remark}

\newcommand{\F}{\mathbb{F}}

\newcommand{\Z}{\mathbb{Z}}

\DeclareMathOperator{\Tr}{Tr}

\newcommand{\mybar}[1]{#1\llap{$\overline{\phantom{\rm#1}}$}}
\newcommand{\abs}[1]{\lvert #1 \rvert}

\begin{document}

\title
[Roots of certain polynomials]
{Roots of certain polynomials over finite fields}

\author{Zhiguo Ding}
\address{
  Hunan Institute of Traffic Engineering,
  Hengyang, Hunan 421001 China
}
\email{ding8191@qq.com}

\author{Michael E. Zieve}
\address{
  Department of Mathematics,
  University of Michigan,
  530 Church Street,
  Ann Arbor, MI 48109-1043 USA
}
\email{zieve@umich.edu}
\urladdr{http://www.math.lsa.umich.edu/$\sim$zieve/}

\thanks{The authors thank Lijing Zheng for sharing a preliminary version of \cite{ZKZPL}, and thank Faruk G\"olo\u{g}lu for valuable correspondence.}


\date{\today}

\begin{abstract}
We determine the roots in $\F_{q^3}$ of the polynomial $X^{2q^\ell+1}+X+a$ for each positive integer $\ell$ and each $a\in\F_q$, where $q$ is a power of $2$. 
We introduce a new approach for this type of question, and we obtain results which are more explicit than the previous results in this area. 
Our results resolve an open problem and a conjecture of Zheng, Kan, Zhang, Peng, and Li.
\end{abstract}

\begin{NoHyper}
\maketitle
\end{NoHyper}


\section{Introduction}

In case $q$ and $Q$ are powers of $2$, the roots in $\F_q$ of polynomials of the form $X^{Q+1}+X+a$ has attracted much attention. 
For instance, the number of such roots is studied in \cite{Bluhermain,BluherAPNBC,BluherDicksonidentity,BTT,Dillonexc,DFHR,vzGGZ,GologluPP,GologluAPNbiprojective,GKLWZ,GKZdiscrete,HHKZLJ,HK,HK2,HKN,HKZLJ,HZ,KCM,KM,MS,QTL,ST,Ugolini,ZLHKasami}, 
and this number has been applied to coding theory \cite{BH,HZ,LTW,XCQ}, APN and related functions in cryptography and combinatorics \cite{BP,BluherAPNBC,BTT,BC,GologluAPNbiprojective,KZ,LHXZ,QTL,Taniguchi}, 
division rings and combinatorial designs \cite{BBKGMP,Dillonexc,GK,HuK,Knuth,KMW,Tang,XLW,XCQ}, cross-correlation of $m$-sequences \cite{DFHR,HHKZLJ,HKN,HKZLJ,JHZ,LTW,ZLHKasami}, 
dynamics over finite fields \cite{Ugolini}, non-uniqueness of functional decomposition of polynomials \cite{vzG,vzGGZ}, permutation polynomials and rational functions \cite{GologluPP,KMCLLJ,KMKJ,ZKZPL}, 
and computation of discrete logarithms in multiplicative groups \cite{GGMZ,GGMZ2,GKLWZ,GKZ2,GKZdiscrete,GKZindiscreet}, elliptic curves \cite{GKZec}, and Jacobians of algebraic curves \cite{Massierer}. 
The roots of $X^{Q+1}+X+a$ (rather than just their number) were studied in \cite{GologluAPNtrihex,GologluPP,HK2,KCM,KCMcomplete,KM,MTW,MKCT}.
In particular, several of the above papers reduce the question of determining the number of roots (or exhibiting the roots) of polynomials of the form $X^{Q+1}+X+a$ to the study of properties of an associated recursively defined sequence of polynomials.
In this paper we determine the roots of a certain class of polynomials of this form, obtaining descriptions for both the roots and the number of roots which are much more explicit than those in previous papers. 
In particular, we find an unexpected connection with Dickson polynomials.

Our results use the following notation:

\begin{itemize}
\item $\Tr$ denotes the trace relative to the field extension $\F_q/\F_2$;
\item for any positive integer $n$, $D_n(X)$ is the Dickson polynomial of the first kind of degree $n$ with parameter $1$, which is determined by the functional equation $D_n(X+X^{-1})=X^n+X^{-n}$ \cite{ACZ,LMT}.
\end{itemize}

We first describe the number of roots in certain difficult cases.

\begin{thm} \label{main}
Let $\ell$ and $m$ be positive integers with $3\nmid\ell$, and write $q:=2^m$ and $n:=\lfloor(q+1)/3\rfloor$. 
Pick any $a\in\F_q\setminus\F_2$ such that $\Tr(1/a)=\Tr(1)$, and let $N_\ell$ be the number of roots in\/ $\F_{q^3}$ of $X^{2q^\ell+1}+X+a$. Then
\begin{enumerate}
\item if $\ell\not\equiv -m\pmod 3$ then $N_\ell=3$;
\item if $\ell\equiv -m\pmod 3$ then $N_\ell\in\{0,9\}$, where $N_\ell=9$ if and only if $D_n(a)=0$.
\end{enumerate}
\end{thm}

The next result (which is easy) provides further information about roots of the Dickson polynomials occurring in Theorem~\ref{main}.

\begin{prop} \label{supplement}
Let $m$ be a positive integer with $3\nmid m$, and write $q:=2^m$ and $n:=\lfloor(q+1)/3\rfloor$. Then
\begin{enumerate}
\item the roots of $D_n(X)$ in\/ $\F_q\setminus\F_2$ are the elements $\zeta+\zeta^{-1}$ where $\zeta\in\F_{q^2}$ and $\zeta^n=1$ but $\zeta\ne 1$;
\item $D_n(X)$ has exactly $\lfloor q/6\rfloor$ roots in\/ $\mybar\F_q\setminus\F_2$, all of which are in\/ $\F_q\setminus\F_2$.
\end{enumerate}
\end{prop}

We now exhibit the roots of the polynomials in Theorem~\ref{main}.  

\begin{thm} \label{roots}
With notation as in Theorem~\emph{\ref{main}}, let $\Gamma_\ell$ be the set of roots in\/ $\F_{q^3}$ of $X^{2q^\ell+1}+X+a$, and pick elements $\omega\in\F_4\setminus\F_2$ and $b\in\F_q$ such that $b^2+b=a^{-1}+1$. 
Then $b\notin\F_4$. Write $c:=(b+\omega)/(b+\omega^2)$, and define
\begin{align*}
\Lambda_0&:=\{v+v^{-1}\colon v\in\F_{q^6}\text{ with } v^3=c\}, \\
\Lambda_1&:=\Bigl\{\frac{a^{-1}+v+v^{-1}}{b^2}\colon v\in\F_{q^6}\text{ with } v^3=\omega c\Bigr\}, \\
\Lambda_2&:=\Bigl\{\frac{a^{-1}+v+v^{-1}}{b^2+1}\colon v\in\F_{q^6}\text{ with } v^3=\omega^2 c\Bigr\},
\end{align*}
and $\Lambda:=\Lambda_0\cup\Lambda_1\cup\Lambda_2$. Then the $\Lambda_i$ are pairwise disjoint sets of size $3$.
\begin{enumerate}
\item If $m\equiv 0\pmod 3$ and $c$ is a cube in\/ $\F_{q^2}$ then $\Gamma_\ell=\Lambda_0$.
\item If $m\equiv 0\pmod 3$ and $c$ is not a cube in\/ $\F_{q^2}$ then $\Gamma_\ell=\Lambda_k$ for the unique $k\in\{1,2\}$ such that $c^{(q^2-1)/3}=\omega^{-k\ell}$.
\item If $m\not\equiv 0\pmod 3$ then there is a unique $k\in\{0,1,2\}$ for which $\omega^k c$ is a cube in\/ $\F_{q^2}$; this value $k$ satisfies the following:
\begin{itemize}
\item if $k=0$ and $\ell\equiv m\pmod 3$ then $\Gamma_\ell=\Lambda_0$;
\item if $k\ne 0$ and $\ell\equiv m\pmod 3$ then $\Gamma_\ell=\Lambda_{3-k}$;
\item if $k=0$ and $\ell\equiv -m\pmod 3$ then $\Gamma_\ell=\Lambda$;
\item if $k\ne 0$ and $\ell\equiv -m\pmod 3$ then $\Gamma_\ell=\emptyset$.
\end{itemize}
\end{enumerate}
\end{thm}

For completeness, we also determine the roots in $\F_{q^3}$ of $X^{2q^\ell}+X+a$ for all elements $a\in\F_q$ not addressed in Theorem~\ref{main}, and also for all integers $\ell$ which are divisible by $3$. 
These remaining cases are much easier than the above results. See Remark~\ref{a=0} (for $a=0$), Proposition~\ref{main2} (for $\Tr(1/a)\ne\Tr(1)$), Corollary~\ref{3} (for $3\mid\ell$), and Corollary~\ref{a=1} (for $a=1$).

We deduce the following consequences of our results for roots of certain related polynomials in $\F_{q^3}$ and in the set $\mu_{q^2+q+1}$ of all $(q^2+q+1)$-th roots of unity in $\F_{q^3}^*$.

\begin{cor} \label{zheng}
Let $q=2^m$ where $m$ is a positive integer with $m\not\equiv 1\pmod 3$, and pick any $h,e\in \F_q\setminus\F_2$ with $h^3=e^2+e+1$. 
Then the polynomial $X^{2q^2+1} + hX + e$ has exactly three roots in\/ $\F_{q^3}$.  
\end{cor}

\begin{cor} \label{zheng2}
Let $q=2^m$ where $m$ is a positive integer, and pick any $\omega\in\F_4\setminus\F_2$ and any $h,e\in \F_q\setminus\F_2$ with $h^3=e^2+e+1$. 
Then, for any $\ell\in\{1,2\}$, the polynomial $X^{2q^\ell+1} + hX + e$ has roots in $\mu_{q^2+q+1}$ if and only if $(e+\omega)^{(q^2-1)/3}=\omega^{-\ell}$, in which case the roots in $\mu_{q^2+q+1}$ are the three roots of $X^3+h^2 X^2 + (e+1)hX+1$.
\end{cor}

The ``odd $m$" case of Corollary~\ref{zheng} proves Conjecture~3.5 of \cite{ZKZPL}. The case $\ell=2$ of Corollary~\ref{zheng2} resolves Open Problem~1 of \cite{ZKZPL}.  
Moreover, Corollary~\ref{zheng2} provides an explicit form of the final condition in each of \cite[Thm.~3.2]{ZKZPL} and \cite[Thm.~3.6]{ZKZPL}; the desire to obtain explicit forms of those results was the original motivation for Open Problem~1 and Conjecture~3.5 of \cite{ZKZPL}. 
We note that \cite[Prop.~3.4]{ZKZPL} determines when the polynomial in Corollary~\ref{zheng} has three roots in the subfield $\F_q$ of $\F_{q^3}$ in case $m$ is odd, and also asserts that the polynomial in Corollary~\ref{zheng2} has no roots in $\mu_{q^2+q+1}$ when $\ell=2$, $m$ is congruent to $3$ or $5$ (mod~$6$), and $1+e\omega$ is a cube in $\F_{q^2}$.
The first assertion in this result is contained in our Proposition~\ref{zhengprop}, and the second follows from Corollary~\ref{zheng2}.

Our proofs of the above results are nearly self-contained, and use methods that are quite different from those that have been used previously. In particular, one key to our approach is the study of the rational function $a/(X^2+1)$.  
It seems conceivable that our approach might also yield explicit results for other instances of the general problems of either counting or naming the roots in $\F_q$ of $X^{Q+1}+X+a$.

This paper is organized as follows. In the next section we prove some easy preliminary results. Then in Section~\ref{sec:notation} we provide the notation used in Sections~\ref{sec:reduce}--\ref{sec:roots}.
In Section~\ref{sec:reduce} we reduce the study of roots in $\F_{q^3}$ of $X^{2q^{\ell}+1}+X+a$ to the study of roots of three associated degree-$3$ polynomials, and in particular we prove Lemma~\ref{twistdeg3} which is crucial to our approach. 
In the next section we prove Theorem~\ref{main} and Proposition~\ref{supplement}, and then in Section~\ref{sec:roots} we prove Theorem~\ref{roots}. 
We conclude the paper in Section~\ref{sec:zheng} by proving refinements of Corollaries~\ref{zheng} and \ref{zheng2}, in which in addition to counting the number of roots of the prescribed polynomials we determine these roots explicitly.


\section{Preliminary results}

In this section we provide some easy preliminary results. 

\begin{lemma} \label{gcd}
Let $R$ be a unique factorization domain, and let $\alpha$ be a nonzero non-unit in $R$. For any positive integers $m$ and $n$, the element $\alpha^{\gcd(m,n)}-1$ is a greatest common divisor of $\alpha^m-1$ and $\alpha^n-1$ in $R$.
\end{lemma}

This result is well-known in case $R=\Z$, but we will also use it in case $R=\F_q[X]$. 
Since we do not know a reference for the latter case, we include the following proof.

\begin{proof}
An element $\beta\in R$ divides both $\alpha^m-1$ and $\alpha^n-1$ if and only if the element $\mybar \alpha:=\alpha+\beta R$ of the quotient ring $R/\beta R$ satisfies $\mybar \alpha^m=1$ and $\mybar \alpha^n=1$. 
This says that the order of $\mybar \alpha$ divides both $m$ and $n$, or equivalently divides $\gcd(m,n)$, i.e., $\mybar \alpha^{\gcd(m,n)}=1$. 
Thus $\beta$ divides both $\alpha^m-1$ and $\alpha^n-1$ if and only if $\beta$ divides $\alpha^{\gcd(m,n)}-1$, which concludes the proof.
\end{proof}

We now state a general result about the number of roots in $\F_q$ of a polynomial of the form $X^{Q+1} + X + a \in \F_q[X]$, where $Q$ is a power of the characteristic of $\F_q$. 
This result has overlap with \cite[Thm.~4.3]{Bluhermain}, \cite[Lemma~22]{DFHR}, \cite[Rmk.~5.14]{GologluPP}, \cite[Lemma~III.3]{GologluAPNbiprojective}, \cite[Thm.~1]{HK}, \cite[Lemma~9]{HZ}, and \cite[Thm.~8]{MS}; we provide a short self-contained proof for the reader's convenience.

\begin{lemma} \label{trinomial}
Write $q:=p^m$ and $Q:=p^n$ where $p$ is prime and $m$ and $n$ are positive integers. For any $a\in \F_q^*$, the number of roots in\/ $\F_q$ of the polynomial $X^{Q+1} + X + a$ is in $\{ 0, 1, 2, 1+p^{\gcd(m,n)} \}$.
Moreover, if $p=2$ and $\gcd(m,n)=1$ then this number of roots is in $\{ 0, 1, 3 \}$.
\end{lemma}

\begin{proof}
Suppose $\lambda\in\F_q$ is a root of $B(X):= X^{Q+1} + X + a$. Then 
\[
B(X) = (X-\lambda)^{Q+1} + \lambda (X-\lambda)^Q + (\lambda^Q+1) (X-\lambda),
\]
so for $\delta\in\F_q\setminus\{\lambda\}$ we have $B(\delta)=0$ if and only if $1/(\delta-\lambda)$ is a root of $\widetilde{B}(X) := 1 + \lambda X + (\lambda^Q+1) X^Q$. 
Since $\widetilde{B}(0)\ne 0$, it follows that $N(B)=N(\widetilde{B}) + 1$, where $N(H)$ denotes the number of the roots in $\F_q$ of a polynomial $H(X)\in\F_q[X]$. 
Since $L(X):=\widetilde{B}(X)-1$ induces a homomorphism from the additive group of $\F_q$ to itself, we have $N(\widetilde B)\in\{0,N(L)\}$. 
Since $B(\lambda)=0$ and $a\ne 0$, we have $\lambda\notin\{0,-1\}$, so that $\lambda':=\lambda/(\lambda^Q+1)$ is in $\F_q^*$ and 
\[
N(L)=N(X^Q+\lambda' X) = 1 + N(X^{Q-1}+\lambda').
\]
Since $X^{Q-1}$ induces a homomorphism from $\F_q^*$ to itself, we have $N(X^{Q-1}+\lambda')\in\{0,N(X^{Q-1}-1)\}$. Finally, two applications of Lemma~\ref{gcd} yield 
\[
\gcd(X^{q-1}-1,X^{Q-1}-1)=X^{\gcd(q-1,Q-1)}-1=X^{p^{\gcd(m,n)}-1}-1, 
\]
so that $N(X^{Q-1}-1)=p^{\gcd(m,n)}-1$. We conclude that $N(B)\in\{1,2,1+p^{\gcd(m,n)}\}$. 
Finally, if $p^{\gcd(m,n)}=2$ then $X^{Q-1}$ permutes $\F_q^*$ so that $N(X^{Q-1}+\lambda')=1$, whence $N(B)\in\{1,3\}$.
\end{proof}

We will use the following special case of the above result.

\begin{cor} \label{Ni}
Let $q:=2^m$, and pick any $a\in \F_{q^3}^*$ and any positive integer $\ell$ coprime to $3$. Then the number of roots in $\F_{q^3}$ of the polynomial $X^{2q^\ell+1}+X+a$ is in $\{ 0, 1, 2, 9 \}$ if $m\equiv -\ell\pmod 3$, and is in $\{0, 1, 3\}$ otherwise.
\end{cor}

\begin{proof}
For $r:=3m$ and $s:=1+\ell m$, the value $\gcd(r,s)=\gcd(3,s)$ equals $3$ if $m\equiv -\ell\pmod 3$, and equals $1$ otherwise. 
Thus Corollary~\ref{Ni} follows from the special case of Lemma~\ref{trinomial} with $p=2$ and with these values of $r$ and $s$.
\end{proof}

We also use the following result on factorizations of cubic polynomials over $\F_{2^m}$ (e.g., cf.\ \cite[Thm.~1]{Williams}).

\begin{lemma} \label{W}
Let $q$ be a power of $2$, and put $f(X):=X^3+aX+b$ where $a\in\F_q$ and $b\in\F_q^*$.
Let $N$ be the number of distinct roots of $f(X)$ in\/ $\F_q$, write $\Tr$ for the trace from\/ $\F_q$ to\/ $\F_2$, and pick $e\in\F_{q^2}^*$ satisfying $e^2+be+a^3=0$.
Then $N\in\{0,1,3\}$, and $N=1$ if and only if and only if\/ $\Tr(a^3/b^2)\ne\Tr(1)$. Moreover, if $N\in\{0,3\}$ then $N=3$ if and only if $e$ is a cube in\/ $\F_{q^2}$.
\end{lemma}

Next we describe the roots of $X^3+X+a$ in $\mybar\F_2$.

\begin{lemma}\label{f0roots}
For any $a,e\in\mybar\F_2^*$ with $e^2+ae+1=0$, the set of roots of $X^3+X+a$ in\/ $\mybar\F_2$ is $\{v+v^{-1}\colon v^3=e\}$.
\end{lemma}

\begin{proof}
We simply check that if $v_0^3=e$ and $\omega\in\F_4\setminus\F_2$ then
\[
X^3+X+v_0^3+v_0^{-3} = \prod_{i=0}^2 \bigl(X+\omega^i v_0 + \omega^{-i} v_0^{-1}\bigr).\qedhere
\]
\end{proof}

\begin{rmk}
Lemma~\ref{W} follows easily from Lemma~\ref{f0roots}, which yields a new proof of Lemma~\ref{W} that is more elementary than the proof in \cite{Williams} (which relies on Berlekamp's characteristic $2$ analogue of Stickelberger's theorem on the parity of the number of irreducible factors of a polynomial over a finite field). 
Also, since any degree-$3$ polynomial over $\mybar\F_2$ may be reduced to one of the forms $X^3+a$ or $X^3+X+a$ by composing with a degree-one polynomial on the right and a scalar multiple on the left, Lemma~\ref{f0roots} yields a similar description of the roots of any degree-$3$ polynomial over $\mybar\F_2$ (and both this description and Lemma~\ref{f0roots} remain valid if $\mybar\F_2$ is replaced by any algebraically closed field of characteristic $2$).
\end{rmk}

We conclude this section by determining the roots in $\F_{q^3}$ of $X^{2q^\ell}+X+a$ in some relatively easy cases.

\begin{cor} \label{3}
Let $q=2^m$, and let\/ $\Tr$ be the trace from\/ $\F_q$ to\/ $\F_2$. For any $a\in\F_q^*$ with\/ $\Tr(1/a)=\Tr(1)$, and any nonnegative integer $\ell$ divisible by $3$, pick $e\in\F_{q^2}^*$ with $e^2+ae+1=0$. 
Let $\Gamma$ be the set of roots in\/ $\F_{q^3}$ of $X^{2q^\ell+1}+X+a$. Then $\abs{\Gamma}=3$, and $\Gamma=\{v+v^{-1}\colon v^3=e\}$.
\end{cor}

\begin{proof}
For $u\in\F_{q^3}$ we have $u^{2q^\ell+1}=u^3$, so $\Gamma$ is the set of roots in $\F_{q^3}$ of $X^3+X+a$. 
Thus Lemma~\ref{W} implies that $\abs{\Gamma}=3$, and Lemma~\ref{f0roots} yields the description of $\Gamma$.
\end{proof}

\begin{prop} \label{main2}
Let $q=2^m$, and pick any $a\in\F_q\setminus\F_2$ such that $\Tr(1/a)\ne\Tr(1)$, where $\Tr$ denotes the trace from\/ $\F_q$ to\/ $\F_2$. 
Then $X^{2q^\ell+1}+X+a$ has a unique root in\/ $\F_{q^3}$ for each nonnegative integer $\ell$. 
This root is $e^n+e^{-n}$ where $e\in\F_{q^2}$ satisfies $e^2+ae=1$ and $n$ is either $(2q-1)/3$ or $(q+2)/3$ according as $m$ is either odd or even.
\end{prop}

\begin{proof}
The roots in $\F_q$ of $H(X):=X^{2q^\ell+1}+X+a$ are precisely the roots in $\F_q$ of $G(X):=X^3+X+a$. Since $\Tr(1/a)\ne 1$, we have $\Tr(1/a^2)\ne 1$, so Lemma~\ref{W} implies that $G(X)$ has a unique root in $\F_q$.  
Hence the number of roots in $\F_{q^3}$ of $H(X)$ is congruent to $1$ mod $3$, so it must be $1$ by Corollary~\ref{Ni}.

It remains to determine the unique root. Pick $e\in\F_{q^2}$ such that $e^2+ae=1$, and write $v:=e^n$ where $n$ is as in the result. 
Then $u:=v+v^{-1}$ satisfies $u^3+u+a=v^3+v^{-3}+e+e^{-1}$, which would be $0$ if $v^3=e$. In order to show that $u$ is a root in $\F_{q^3}$ of $X^{2q^\ell+1}+X+a$, it suffices to show that $u\in\F_q$ and $v^3=e$. 
If $m$ is odd then $\Tr(1/a^2)=\Tr(1/a)=1+\Tr(1)=0$, so since $(e/a)^2+(e/a)=1/a^2$ we conclude that $e/a\in\F_q$ and thus $e\in\F_q$; hence $u\in\F_q$ and $v^3=e^{3n}=e^{2q-1}=e$, as required. Now assume $m$ is even, so that $\Tr(1/a^2)=1+\Tr(1)=1$, and thus $e\notin\F_q$. 
Then $X^2+aX+1$ is the minimal polynomial of $e$ over $\F_q$, so the roots of this polynomial are $e$ and $e^q$, and thus $e^{q+1}=1$. It follows that 
\[
u^q=v^q+v^{-q}=e^{nq}+e^{-nq}=e^{-n}+e^n=v^{-1}+v=u, 
\]
so that $u\in \F_q$, and also $v^3=e^{3n}=e^{q+2}=e$. Thus in each case $u$ is in $\F_q$ and $u^{2q^\ell+1}+u+a=0$, as desired.
\end{proof}

\begin{rmk} \label{a=0}
If $q$ and $Q$ are powers of $2$ then the roots in\/ $\F_q$ of $X^{Q+1}+X$ are $0$ and $1$, since $X^{Q+1}+X=X(X+1)^Q$.
\end{rmk}


\section{Notation} \label{sec:notation}

In the next three sections we use the following notation:

\begin{itemize}
\item $\ell$ is a prescribed positive integer coprime to $3$,
\item $q:=2^m$ for some positive integer $m$,
\item $\mybar\F_q$ is the algebraic closure of $\F_q$,
\item $\Tr(X)$ is the trace relative to the field extension $\F_q/\F_2$,
\item $a\in\F_q^*$ satisfies $\Tr(1/a)=\Tr(1)$ (and $a\ne 1$ after section~\ref{sec:reduce}),
\item $b$ is a prescribed element of $\F_q$ satisfying $b^2+b+1=1/a$ (and $b\notin\F_4$ after section~\ref{sec:reduce}),
\item $\omega$ is a prescribed element of $\F_4\setminus\F_2$,
\item if $a\ne 1$ then $c:=(b+\omega)/(b+\omega^2)$ (which is only used after section~\ref{sec:reduce}),
\item $\rho(X):= a/(X^2+1)$,
\item $f(X):= (a^2+1)X^9 +aX^8 + (a^4+a^2+1)X + (a^5+a)$,
\item $f_0(X):= X^3 + X + a$,
\item $f_1(X) := ab^2 X^3 + X^2 + a(b+1)^2 X + a^2 b^4$,
\item $f_2(X):= a(b+1)^2 X^3 + X^2 + ab^2 X + a^2(b+1)^4$,
\item $H_k(X):=X^{2q^k+1}+X+a$ for any nonnegative integer $k$,
\item $\Gamma_k$ is the set of roots in $\F_{q^3}$ of $H_k(X)$,
\item $N_k$ is the size of $\Gamma_k$,
\item for any positive integer $n$, $D_n(X)$ is the Dickson polynomial of the first kind with degree $n$ and parameter $1$, which is the unique polynomial in $\F_2[X]$ satisfying $D_n(X+X^{-1})=X^n+X^{-n}$.
\end{itemize}


\section{From roots of $H_\ell(X)$ to irreducibility of $f_i(X)$} \label{sec:reduce}

In this section we prove the following result, using the notation from Section~\ref{sec:notation}.

\begin{prop} \label{fi}
Suppose $a\ne 1$. For $k\in\{1,2\}$ and any irreducible degree-$3$ polynomial $g(X) \in\F_q[X]$ which divides $H_k(X)$, there is a unique $i\in\{1,2\}$ for which $g(X)$ is a constant times $f_i(X)$. 
Conversely, for $i\in\{1,2\}$, if $f_i(X)$ is irreducible over\/ $\F_q$ then $f_i(X)$ divides $H_k(X)$ for a unique $k\in\{1,2\}$.
\end{prop}

Although the above result requires $a\ne 1$, we allow $a=1$ in the first two lemmas below, since we will use these lemmas to resolve the case $a=1$ in Corollary~\ref{a=1}.

\begin{lemma} \label{numerator}
We have
\[
\bigl( \rho(X)\circ \rho(X)\circ \rho(X) \bigr) + X = \frac{f(X)}{g(X)}
\]
for some $g(X)\in\F_q[X]$ which is coprime to $f(X)$.
\end{lemma}

\begin{proof}
We compute
\begin{align*}
\rho(X) \circ \rho(X) \circ \rho(X)
&= \frac{a}{X^2+1} \circ \frac{a}{X^2+1} \circ \frac{a}{X^2+1} \\ 
&= a\frac{X^4+1}{X^4+(a^2+1)} \circ \frac{a}{X^2+1} \\
&= a\frac{a^4+X^8+1}{a^4+(a^2+1)(X^8+1)} \\
&= \frac{h(X)}{g(X)}
\end{align*}
where $g(X):=(a^2+1)X^8+(a^4+a^2+1)$ and $h(X):=aX^8+(a^5+a)$. 
Here $g(X)$ and $h(X)$ are coprime since $ag(X)+(a^2+1)h(X)=a^7$ is a nonzero constant (or alternately, since $\max\bigl(\deg(g),\deg(h)\bigr)=8=\deg(\rho\circ\rho\circ\rho)$). Thus
\[
\bigl(\rho(X)\circ\rho(X)\circ\rho(X)\bigr)+X=\frac{Xg(X)+h(X)}{g(X)}
\]
where $Xg(X)+h(X)$ and $g(X)$ are coprime, which concludes the proof since $Xg(X)+h(X)=f(X)$.
\end{proof}

\begin{lemma} \label{twistdeg3}
The following statements hold for each $\beta\in\mybar\F_q$.
\begin{itemize}
\item We have $\beta\in\F_{q^3}$ and $H_2(\beta)=0$ if and only if $f(\beta)=0$ and $\rho(\beta)=\beta^q$.
\item We have $\beta\in\F_{q^3}$ and $H_1(\beta)=0$ if and only if $f(\beta)=0$ and $\rho(\beta)=\beta^{1/q}$.
\item We have $H_0(\beta)=0$ if and only if $f(\beta)=0$ and $\rho(\beta)=\beta$.
\end{itemize}
\end{lemma}

\begin{proof}
For $\beta\in\F_{q^3}$ we have $H_2(\beta)=0$ if and only if $(H_2(\beta))^q=a$, i.e., $\beta^{q+2}+\beta^q=a$, or equivalently $\beta^q=a/(\beta^2+1)=\rho(\beta)$.
For $\beta\in\mybar\F_q$ such that $\rho(\beta)=\beta^q$, we have
\[
( \rho\circ\rho\circ\rho ) (\beta) = ( \rho\circ\rho ) (\beta^q) = \rho(\beta^{q^2}) = \beta^{q^3},
\]
so Lemma~\ref{numerator} implies that $f(\beta)=0$ if and only if $\beta\in\F_{q^3}$.
Thus the roots $\beta\in\mybar\F_q$ of $f(X)$ which satisfy $\rho(\beta)=\beta^q$ are precisely the elements $\beta\in\F_{q^3}$ satisfying $\rho(\beta)=\beta^q$, which are the roots in $\F_{q^3}$ of $H_2(X)$.

For $\beta\in\F_{q^3}$ we have $H_1(\beta)=0$ if and only if $\beta=a/(\beta^{2q}+1)$, which upon taking $q$-th roots becomes $\beta^{1/q}=a/(\beta^2+1)=\rho(\beta)$. 
For $\beta\in\mybar\F_q$ such that $\rho(\beta)=\beta^{1/q}$, we have
\[
( \rho\circ\rho\circ\rho ) (\beta) = ( \rho\circ\rho ) (\beta^{1/q}) = \rho(\beta^{1/q^2}) = \beta^{1/q^3},
\]
so Lemma~\ref{numerator} implies that $f(\beta)=0$ if and only if $\beta\in\F_{q^3}$.
Thus the roots $\beta\in\mybar\F_q$ of $f(X)$ which satisfy $\rho(\beta)=\beta^{1/q}$ are precisely the elements $\beta\in\F_{q^3}$ satisfying $\rho(\beta)=\beta^{1/q}$, which are the roots in $\F_{q^3}$ of $H_1(X)$.

Finally, for $\beta\in\mybar\F_q$ we have $H_0(\beta)=0$ if and only if $\beta=a/(\beta^2+1)=\rho(\beta)$, in which case $\beta$ is fixed by $\rho(X)\circ\rho(X)\circ\rho(X)$ so that $f(\beta)=0$.
\end{proof}

We now treat the case $a=1$.

\begin{cor} \label{a=1}
In case $a=1$, we have $N_\ell=3$ if $3\mid m (\ell-m)$, and $N_\ell=0$ otherwise. Moreover, if $3\mid m$ then $\Gamma_\ell$ is the set of roots of $X^3+X+1$, and if $\ell\equiv m\equiv\pm 1\pmod 3$ then $\Gamma_\ell$ is the set of roots of $X^3+X^2+1$.
\end{cor}

\begin{proof}
First suppose that $\beta\in\F_{q^3}$ is a root of $H_\ell(X)$, so that $\beta\notin\F_2$. 
By Lemma~\ref{twistdeg3} we know that $\beta$ is a root of $f(X)$, which is $X^8+X$ since $a=1$. 
Thus $\beta\in \F_8\setminus\F_2$, so that $\beta$ is a primitive $7$-th root of unity and $(\beta^3+\beta+1)(\beta^3+\beta^2+1)=0$.

Conversely, if $\beta^3+\beta=1$ then $H_\ell(\beta)=0$ if and only if $\beta^{2q^\ell+1}=\beta^3$, or equivalently $\beta^{q^\ell-1}=1$. 
Since $\beta$ has order $7$, this says $2^{\ell m}\equiv 1\pmod{7}$, or equivalently $3\mid m$.

If $\beta^3+\beta^2=1$ then $H_\ell(\beta)=0$ if and only if $H_\ell(\beta)^2=0$, or equivalently $\beta^{4q^\ell+2}=\beta^3$, i.e., $\beta^{4q^\ell-1}=1$. 
Since $\beta$ has order $7$, this says $2^{\ell m+2}\equiv 1\pmod{7}$, or equivalently $\ell m\equiv 1\pmod{3}$, i.e., $\ell\equiv m\equiv\pm 1\pmod{3}$.

Since all order-$7$ elements in $\mybar\F_q^*$ are contained in $\F_{q^3}$, the result follows.
\end{proof}

\begin{lemma} \label{factorization}
Suppose $a\ne 1$. Then $f(X) = f_0(X) f_1(X) f_2(X)$, and $f(X)$ has nine distinct roots in\/ $\mybar\F_q$. 
Moreover, $\rho(X)$ fixes each root of $f_0(X)$, and $\rho(X)$ acts as a $3$-cycle on the roots of $f_i(X)$ for $i\in\{1,2\}$.
\end{lemma}

\begin{proof}
It is routine to verify that $f(X) = f_0(X) f_1(X) f_2(X)$ and $f(X)$ is a degree-$9$ polynomial with no multiple roots, so by Lemma~\ref{twistdeg3} we conclude that $\rho(X)$ fixes all roots of $f_0(X)$ but does not fix any roots of $f_i(X)$ when $i\in\{1,2\}$. 
Moreover, one can check that for each $i\in\{1,2\}$ the polynomial $f_i(X)$ divides the numerator of $f_i(X)\circ\rho(X)$, so that $\rho(X)$ maps each root of $f_i(X)$ to a root of $f_i(X)$. 
By Lemma~\ref{numerator}, the roots of $f(X)$ are fixed by $\rho(X)\circ\rho(X)\circ\rho(X)$, so that $\rho(X)$ induces a $3$-cycle on the roots of $f_i(X)$ for each $i\in\{1,2\}$.
\end{proof}

\begin{lemma} \label{f0}
If $f_0(X)$ is irreducible over\/ $\F_q$ then $f_0(X)$ does not divide $H_k(X)$ for any $k\in\{1,2\}$.
\end{lemma}

\begin{proof}
Suppose otherwise. Then $f_0(X)$ divides both $X^{q^3}-X$ and $G(X):=f_0(X)+H_k(X)=X^{2q^k+1}+X^3=X^3(X^{q^k-1}-1)^2$.
Thus $f_0(X)$ divides $\gcd(G(X),X^{q^3}-X)=X\gcd(X^{q^k-1}-1,X^{q^3-1}-1)$, which is $X^q-X$ by two applications of Lemma~\ref{gcd}. But this contradicts irreducibility of $f_0(X)$.
\end{proof}

\begin{rmk}
Alternately, Lemma~\ref{f0} may be deduced from Lemma~\ref{twistdeg3}.
\end{rmk}

We now prove Proposition~\ref{fi}.

\begin{proof}[Proof of Proposition~\emph{\ref{fi}}]
First let $g(X)$ be an irreducible degree-$3$ polynomial in $\F_q[X]$ which divides $H_k(X)$ for some $k\in\{1,2\}$. Then $g(X)$ divides $f(X)$ in light of Lemma~\ref{twistdeg3}. 
By Lemma~\ref{factorization}, $f(X)=f_0(X)f_1(X)f_2(X)$ where the $f_i(X)$'s are pairwise coprime. By definition, each $f_i(X)$ is a degree-$3$ polynomial in $\F_{q^2}[X]$. 
Since $g(X)$ is irreducible over $\F_q$ and $\deg(g)$ is odd, we see that $g(X)$ is also irreducible over $\F_{q^2}$, so there is a unique $i\in\{0,1,2\}$ for which $g(X)$ divides $f_i(X)$, and thus $g(X)$ is a constant times $f_i(X)$. Finally, Lemma~\ref{f0} shows that $i\ne 0$.

Conversely, pick $i\in\{1,2\}$ and suppose that $f_i(X)$ is irreducible over $\F_q$. By Lemma~\ref{factorization}, $f_i(X)$ divides $f(X)$, and also $\rho(X)$ acts as a $3$-cycle on the roots of $f_i(X)$. 
Since the two $3$-cycles on the roots of $f_i(X)$ are induced by the $q$-th power map and the $1/q$-th power map, it follows that $\rho(X)$ acts on the roots of $f_i(X)$ in the same way as exactly one of these two maps. 
Then Lemma~\ref{twistdeg3} implies that $f_i(X)$ divides $H_k(X)$ for exactly one $k\in\{1,2\}$.
\end{proof}


\section{Proof of Theorem~\ref{main}}

In this section we prove Theorem~\ref{main}. We use the notation from Section~\ref{sec:notation}, in addition to requiring $a\ne 1$; since $b^2+b+1=1/a$, it follows that $b\notin\F_4$. In particular we have $c:=(b+\omega)/(b+\omega^2)$.

In light of Proposition~\ref{fi}, we first determine how the polynomials $f_i(X)$ factor over $\F_q$.

\begin{lemma} \label{irr}
For any $i\in\{0,1,2\}$, the polynomial $f_i(X)$ has three distinct roots in\/ $\F_q$ if $\omega^i c$ is a cube in\/ $\F_{q^2}$, and $f_i(X)$ is irreducible in\/ $\F_q[X]$ otherwise.
\end{lemma}

\begin{proof}
Plainly each $f_i(X)$ is a degree-$3$ polynomial in $\F_q[X]$. Write $u_i$ and $v_i$ for the coefficients of $X^3$ and $X^2$ in $f_i(X)$, respectively. 
Then $f_i(X)$ has the same number of roots in $\F_q$ as does $g_i(X):= u_i^{-1} f_i(X+u_i^{-1}v_i)$. Here $g_0(X)=f_0(X)$, and we compute
\[
g_1(X) = X^3 + \frac{1}{b^4} X + \frac{1}{b^4(b^2+b+1)}
\]
and
\[
g_2(X) = X^3 + \frac{1}{b^4+1}X+\frac{1}{(b^4+1)(b^2+b+1)}.
\]
Write $a_i$ and $b_i$ for the coefficients of the terms of $g_i(X)$ of degrees $1$ and $0$, respectively, and write $c_i:=a_i^3/b_i^2$. Then
\begin{align*}
c_0 &= 1+b^2+b^4, \\
c_1 &= 1+\frac{1}{b^2}+\frac{1}{b^4}, \\
c_2 &= 1+\frac{1}{b^2+1}+\frac{1}{b^4+1}.
\end{align*}
Note that $c_i=1+d_i+d_i^2$ where $d_0:=b^2$, $d_1:=1/b^2$, and $d_2:=1/(b^2+1)$. Thus $\Tr(c_i)=\Tr(1)$, so that $f_i(X)$ has either zero or three roots in $\F_q$ by Lemma~\ref{W}. 
Then one root of $h_i(X):=X^2+b_iX+a_i^3$ is $e_i:=b_i(d_i+\omega)$, and we compute
\begin{align*}
e_0 &= \frac{b^2+\omega}{b^2+b+1} = \Big( \frac{b+\omega}{b+\omega^2} \Big)^{-1} = c^{-1}, \\
e_1 &= \frac{1}{b^6}\cdot\frac{\omega b^2+1}{b^2+b+1} = 
\frac{1}{b^6}\cdot\Big(\omega\frac{b+\omega}{b+\omega^2}\Big) = \frac{\omega c}{b^6}, \\
e_2 &= \frac{1}{(b+1)^6}\cdot\frac{\omega b^2+\omega^2}{b^2+b+1} =
\frac{1}{(b+1)^6}\cdot\Big(\omega^2\frac{b+\omega}{b+\omega^2}\Big)^{-1} = \frac{(\omega^2 c)^{-1}}{(b+1)^6}.
\end{align*}
Since $b^6$ and $(b+1)^6$ are nonzero cubes in $\F_{q^2}$, the result follows from Lemma~\ref{W}.
%
%
\end{proof}

\begin{cor} \label{irr2}
The following hold:
\begin{itemize}
\item if $3\mid m$ then either all three polynomials $f_i(X)$ are irreducible over\/ $\F_q$ or all three $f_i(X)$'s have three distinct roots in\/ $\F_q$;
\item If $3\nmid m$ then one of $f_0(X)$, $f_1(X)$, and $f_2(X)$ has three distinct roots in\/ $\F_q$ and the other two $f_i(X)$'s are irreducible over $\F_q$.
\end{itemize}
\end{cor}

\begin{proof}
Note that $3\mid m$ if and only if $6\mid (2m)$, which is equivalent to $q^2\equiv 1\pmod{9}$ since the order of $2$ mod $9$ is $6$. It follows that $3\mid m$ if and only if $\omega$ is a cube in $\F_{q^2}$. 
Thus if $3\mid m$ then the three values $\omega^i c$ with $i\in\{0,1,2\}$ are either all cubes in $\F_{q^2}^*$ or all non-cubes in $\F_{q^2}^*$, which by Lemma~\ref{irr} says that either all three $f_i(X)$'s have three distinct roots in $\F_q$ or all three $f_i(X)$'s are irreducible over $\F_q$. 
Henceforth assume $3\nmid m$. Then $\omega$ is a non-cube in $\F_{q^2}$, so that the three values $\omega^i c$ with $i\in\{0,1,2\}$ lie in three distinct cosets of $\F_{q^2}^*/(\F_{q^2}^*)^3$. 
Hence exactly one of these values is a cube in $\F_{q^2}$, so Lemma~\ref{irr} says that exactly one $f_i(X)$ has three distinct roots in $\F_q$ while the other two $f_i(X)$'s are irreducible over $\F_q$.  
\end{proof}

We now prove the following variant of Theorem~\ref{main}.

\begin{thm} \label{maincubic}
Using the notation of Section~\emph{\ref{sec:notation}}, where in addition we assume $a\ne 1$, both of the following hold:
\begin{enumerate}
\item if $m\not\equiv -\ell\pmod 3$ then $N_\ell=3$;
\item if $m\equiv -\ell\pmod{3}$ then $N_\ell\in\{0,9\}$, where $N_\ell=9$ if and only if $c$ is a cube in\/ $\F_{q^2}$.
\end{enumerate}
Moreover, $\Gamma_\ell$ is the set of roots of $F(X)$, where $F(X)$ is as follows:
\begin{itemize}
\item if $m\equiv 0\pmod 3$ and $c$ is a cube in\/ $\F_{q^2}$ then $F(X)=f_0(X)$;
\item if $m\equiv 0\pmod 3$ and $c$ is not a cube in\/ $\F_{q^2}$ then $F(X)=f_i(X)$ for some $i\in\{1,2\}$;
\item if $m\equiv \ell\pmod 3$ then there is a unique $i\in\{0,1,2\}$ for which $\omega^i c$ is a cube in\/ $\F_{q^2}$; if $i=0$ then $F(X)=f_0(X)$, and otherwise $F(X)=f_{3-i}(X)$;
\item if $N_\ell=9$ then $F(X)=f_0(X)f_1(X)f_2(X)$.
\end{itemize}
\end{thm}

\begin{proof}
First suppose that $m\equiv 0\pmod 3$, so that $N_1,N_2\in\{0,1,3\}$ by Corollary~\ref{Ni}.  
If $c$ is a cube in $\F_{q^2}$ then $f_0(X)$ has three roots in $\F_q$ by Lemma~\ref{irr}; since each such root is a root of $H_\ell(X)$, the condition $N_\ell\le 3$ implies that $N_\ell=3$ and $\Gamma_\ell$ is the set of roots of $f_0(X)$. 
If $c$ is not a cube in $\F_{q^2}$ then, by Lemma~\ref{irr} and Corollary~\ref{irr2}, both $f_1(X)$ and $f_2(X)$ are irreducible over $\F_q$. 
Since $f_1(X)$ is not a constant multiple of $f_2(X)$, Proposition~\ref{fi} implies that in this case $f_1(X)f_2(X)$ divides $H_1(X)H_2(X)$, so that $N_1+N_2\ge 6$.
Since $N_1,N_2\le 3$, it follows that $N_1=N_2=3$, so that $N_\ell=3$, and also that $\Gamma_\ell$ is the set of roots of $f_i(X)$ for some $i\in\{1,2\}$.

Now suppose that $m\equiv k\pmod{3}$ for some $k\in\{1,2\}$. Then Corollary~\ref{Ni} implies that $N_k\in\{0,1,3\}$ and $N_{3-k}\in\{0,1,2,9\}$. 
By Lemma~\ref{irr} and Corollary~\ref{irr2}, there is exactly one $i\in\{0,1,2\}$ for which $\omega^i c$ is a cube in $\F_{q^2}$, and then $f_i(X)$ has three distinct roots in $\F_q$ while $f_j(X)$ is irreducible over $\F_q$ for each $j\ne i$.
Plainly the roots in $\F_q$ of each of $H_1(X)$ and $H_2(X)$ are the roots in $\F_q$ of $f_0(X)$. Thus if $i=0$ then $H_1(X)$ and $H_2(X)$ each have three roots in $\F_q$.
But if $i=0$ then Proposition~\ref{fi} implies that $f_1(X)$ and $f_2(X)$ divide $H_1(X)H_2(X)$. Since $f_1(X)$ is not a constant multiple of $f_2(X)$, it follows that $N_1+N_2\ge 12$, whence $N_k=3$ and $N_{3-k}=9$. 
Moreover, we have shown that $\Gamma_k$ and $\Gamma_{3-k}$ are the sets of roots of $f_0(X)$ and $f_0(X)f_1(X)f_2(X)$, respectively. Henceforth suppose that $i\in\{1,2\}$, so that $f_0(X)$ is irreducible over $\F_q$ and thus $H_1(X)$ and $H_2(X)$ have no roots in $\F_q$.  
By Proposition~\ref{fi}, it follows that the roots in $\F_{q^3}$ of $H_1(X)H_2(X)$ are precisely the roots of $f_{3-i}(X)$, and in addition that $\{N_1,N_2\}=\{0,3\}$.Thus $N_k=3$ and $N_{3-k}=0$, and $\Gamma_k$ is the set of roots of $f_{3-i}(X)$.
\end{proof}

Theorem~\ref{main} follows from the combination of Theorem~\ref{maincubic} and the following fact.

\begin{lemma} \label{Dickson}
The element $c$ is a cube in\/ $\F_{q^2}$ if and only if $D_n(a)=0$, where $n:=\lfloor(q+1)/3\rfloor$.
\end{lemma}

\begin{proof}
If $m$ is odd then $q\equiv 2 \pmod 3$, so $\omega^q = \omega^2$, whence $c = (b+\omega^2)^{q-1}$ so that $c^{q+1}=1$. If $m$ is even then $\omega\in\F_q$, so that $c\in\F_q$. 
Thus for any $m$ we have $c^{3n}=1$, where in addition $\gcd(3n,(q^2-1)/3)=n$. Now $c$ is a cube in $\F_{q^2}$ if and only $c^{(q^2-1)/3}=1$, or equivalently $1=c^{\gcd(3n,(q^2-1)/3)}=c^n$. 
Since $c=(b+\omega)/(b+\omega^2)$, this says $(b+\omega)^n=(b+\omega^2)^n$. Since $b^2+b+1=1/a$ we have $(b+\omega)(b+\omega^2)=1/a$, or equivalently $\sqrt{a}(b+\omega^2) = 1/\sqrt{a}(b+\omega)$. 
So $(b+\omega)^n=(b+\omega^2)^n$ if and only if $d^n=1/d^n$ where $d:=\sqrt{a}(b+\omega)$. Thus $c$ is a cube in $\F_{q^2}$ if and only if $D_n(d+d^{-1})=0$. 
But $d+d^{-1}=\sqrt{a}(b+\omega+b+\omega^2)=\sqrt{a}$, so that $D_n(d+d^{-1})=0$ if and only if $D_n(\sqrt{a})=0$, which upon squaring yields the equivalent condition $D_n(a)=0$.
\end{proof}

We conclude this section with a proof of Proposition~\ref{supplement}.

\begin{proof}[Proof of Proposition~\emph{\ref{supplement}}]
We have $n=(q+1)/3$ if $m$ is even and $n=(q-1)/3$ otherwise. Thus $n\mid (q^2-1)$, and the hypothesis $3\nmid m$ implies that $3\nmid n$. 
Since $D_n(X+X^{-1})=X^n+X^{-n}$, the roots of $D_n(X)$ in $\mybar\F_q$ are the elements $\zeta+\zeta^{-1}$ where $\zeta\in\mybar\F_q^*$ satisfies $\zeta^n=\zeta^{-n}$, or equivalently $\zeta^n=1$. 
For $\zeta\in\mybar\F_q^*$, plainly $\zeta+\zeta^{-1}$ is in $\F_2$ if and only if $\zeta^3=1$. This proves (1). Moreover, if $\zeta^n=1$ and $n=(q-1)/3$ then $\zeta\in\F_q^*$ so also $\zeta+\zeta^{-1}\in\F_q$. 
If $\zeta^n=1$ and $n=(q+1)/3$ then $\zeta^q=\zeta^{-1}$ so that $(\zeta+\zeta^{-1})^q=\zeta+\zeta^{-1}$, whence $\zeta+\zeta^{-1}\in\F_q$. Thus every root in $\mybar\F_q$ of $D_n(X)$ is in $\F_q$. 
Finally, if $\zeta^n=\eta^n=1$ then $\zeta+\zeta^{-1}=\eta+\eta^{-1}$ if and only if $\eta\in\{\zeta,\zeta^{-1}\}$, so the number of roots of $D_n(X)$ in $\mybar\F_q\setminus\F_2$ is $(n-1)/2$. 
This yields (2), since the integer $(n-1)/2$ is either $(q-2)/6$ or $(q-4)/6$, and hence equals $\lfloor q/6\rfloor$.
\end{proof}


\section{Proof of Theorem~\ref{roots}} \label{sec:roots}

In this section we prove Theorem~\ref{roots}. We use the notation from Section~\ref{sec:notation}, where in addition we assume that $a\ne 1$, so that $b\in\F_q\setminus\F_4$ and $c:=(b+\omega)/(b+\omega^2) \in \F_{q^2}^*$. 
In light of Theorem~\ref{maincubic}, there are two main issues to resolve: first, we must exhibit the roots of each $f_i(X)$; and second, in case $3\mid m$ and $c$ is not a cube in $\F_{q^2}$, we must determine which of $f_1(X)$ or $f_2(X)$ divides $H_\ell(X)$.

We first determine the roots of $f_i(X)$. We need only do this for $i\in\{1,2\}$, since we determined the roots of $f_0(X)$ in Lemma~\ref{f0roots}.

\begin{lemma} \label{f1roots}
Define
\begin{align*}
\Lambda_1&:=\Bigl\{\frac{a^{-1}+v+v^{-1}}{b^2}\colon v\in\F_{q^6}\text{ with } v^3=\omega c\Bigr\}, \\
\Lambda_2&:=\Bigl\{\frac{a^{-1}+v+v^{-1}}{b^2+1}\colon v\in\F_{q^6}\text{ with } v^3=\omega^2 c\Bigr\}.
\end{align*}
Then $\Lambda_i$ is the set of roots in\/ $\mybar\F_q$ of $f_i(X)$, for each $i\in\{1,2\}$.
\end{lemma}

\begin{proof}
We simply check that if $v_1^3=\omega c$ then
\[
f_1(X) = ab^2 \prod_{i=0}^2 \Bigl(X + \frac{a^{-1}+\omega^i v_1 + \omega^{-i} v_1^{-1}}{b^2}\Bigr),
\]
and likewise if $v_2^3=\omega^2 c$ then
\[
f_2(X) = a(b+1)^2 \prod_{i=0}^2 \Bigl(X + \frac{a^{-1}+\omega^i v_2 + \omega^{-i} v_2^{-1}}{b^2+1}\Bigr). \qedhere
\]
\end{proof}

Now we address the case that $3\mid m$ and $c$ is not a cube in $\F_{q^2}$.

\begin{lemma} \label{extra}
Suppose that $3\mid m$ and $c$ is not a cube in\/ $\F_{q^2}$. Then $f_1(X)$ divides $H_\ell(X)$ if and only if $c^{(q^2-1)/3}=\omega^{-\ell}$.
\end{lemma}

\begin{proof}
Since $3\mid m$, $\omega$ is a cube in $\F_{q^2}$, so that $\omega c$ and $\omega^2 c$ are non-cubes in $\F_{q^2}$. Write $u:=(a^{-1}+v+v^{-1})/b^2$ where $v^3=\omega c$, so that $v\notin\F_{q^2}$.

First assume $m$ is even, so that $\omega\in\F_q$ and thus $v^3\in\F_q$, whence $v^{3q-3}=1$; since $v\notin\F_q$, it follows that $v^{q-1}=\omega^i$ for some $i\in\{1,2\}$. 
Thus $v^{q^2-1}=\omega^{2i}$, so that $v^{q^\ell}=v\omega^{\ell i}$. Note that $v^{q^\ell-1}=(\omega c)^{(q^\ell-1)/3}=c^{(q^\ell-1)/3}$. We compute
\[
u^{2q^\ell}=\frac{a^{-2}+v^2\omega^{2\ell i}+v^{-2}\omega^{-2\ell i}}{b^4}=\frac{b^4+b^2+1+v^2\omega^{-\ell i}+v^{-2}\omega^{\ell i}}{b^4},
\]
so that $u^{2q^\ell+1}=u+a$ if and only if
\[
(b^2+b+1+v+v^{-1})\cdot (b^4+b^2+1+v^2\omega^{-\ell i}+v^{-2}\omega^{\ell i}) = b^4\cdot (b^2+b+1+v+v^{-1}+ab^2).
\]
It is routine to check that this equality holds if $\ell i\equiv 1\pmod 3$, but if $\ell i\equiv 2\pmod 3$ then the sum of the two sides is
\[
\frac{(v+1)^4 (v+\omega)^2 (v+\omega^2)^6}{v^3 (v^3+\omega)^2},
\]
which is nonzero since $v^3=1$ if and only if $b\omega+\omega^2=b+\omega^2$, which does not occur since $b\ne 0$.  
Thus $H_\ell(u)=0$ if and only if $c^{(q-1)/3}=\omega^\ell$, or equivalently $c^{(q^2-1)/3}=\omega^{-\ell}$.

Next assume $m$ is odd, so that $\omega^q=\omega^2$, and thus $c^q=(b+\omega^2)/(b+\omega)=c^{-1}$, so $c^{q+1}=1$. 
It follows that $v^{3q+3}=1$, and since $v\notin\F_{q^2}$ we conclude that $v^{q+1}=\omega^i$ for some $i\in\{1,2\}$. Hence $v^q=\omega^i v^{-1}$ and $v^{q^2}=\omega^i v$. Thus
\[
u^{2q} = \frac{a^{-2}+\omega^{2i}v^{-2}+\omega^{-2i}v^2}{b^4},
\]
so we conclude as above that $u^{2q+1}=u+a$ if and only if $i\equiv 2\pmod 3$. Likewise,
\[
u^{2q^2} = \frac{a^{-2}+\omega^{2i}v^2+\omega^{-2i}v^{-2}}{b^4},
\]
so that $u^{2q^2+1}=u+a$ if and only if $i\equiv 1\pmod 3$. Hence $H_\ell(u)=0$ if and only if $c^{(q+1)/3}=\omega^{-\ell}$, or equivalently $c^{(q^2-1)/3}=\omega^{-\ell}$.

We have shown that in every case $H_\ell(u)=0$ if and only if $c^{(q^2-1)/3}=\omega^{-\ell}$. Since $u$ varies over the three roots of $f_1(X)$ by Lemma~\ref{f1roots}, it follows that $f_1(X)\mid H_\ell(X)$ if and only if $c^{(q^2-1)/3}=\omega^{-\ell}$.
\end{proof}

\begin{proof}[Proof of Theorem~\emph{\ref{roots}}]
By Lemma~\ref{f1roots}, the sets $\Lambda_1$ and $\Lambda_2$ in Theorem~\ref{roots} are the sets of roots in $\mybar\F_q$ of $f_1(X)$ and $f_2(X)$, respectively.
Moreover, we have $c+c^{-1}=(b^2+b+1)^{-1}=a$, so that $c^2+ac+1=0$ and thus in Corollary~\ref{f0roots} we may put $e:=c$ to conclude that $\Lambda_0$ is the set of roots in $\mybar\F_q$ of $f_0(X)$. 
By Lemma~\ref{factorization}, the $\Lambda_i$ are pairwise disjoint sets of size $3$. Now items (1) and (3) of Theorem~\ref{roots} follow from Theorem~\ref{maincubic}.

Henceforth suppose that $3\mid m$ and $c$ is not a cube in $\F_{q^2}$. Then Theorem~\ref{maincubic} implies that $\Gamma_\ell$ is in $\{\Lambda_1,\Lambda_2\}$.
Finally, Lemma~\ref{extra} shows that $\Gamma_\ell=\Lambda_1$ if and only if $c^{(q^2-1)/3}=\omega^{-\ell}$, so that also $\Gamma_\ell=\Lambda_2$ if and only if $c^{(q^2-1)/3}=\omega^{-2\ell}$, which yields item (2) of Theorem~\ref{roots}.
\end{proof}


\section{Proof of the Open Problem and Conjecture of Zheng et al.} \label{sec:zheng}

In this section we prove refinements of Corollaries~\ref{zheng} and \ref{zheng2}.  Throughout this section we use the following notation:

\begin{itemize}
\item $\ell$ is a prescribed positive integer coprime to $3$,
\item $q:=2^m$ for some positive integer $m$,
\item $\mu_{q^2+q+1}$ is the set of $(q^2+q+1)$-th roots of unity in $\F_{q^3}^*$,
\item $\Tr(X)$ is the trace relative to the field extension $\F_q/\F_2$,
\item $\omega$ is a prescribed element of $\F_4\setminus\F_2$,
\item $h,e\in\F_q\setminus\F_2$ satisfy $h^3=e^2+e+1$,
\item $u:=\sqrt{h}$,
\item $b:=1/\sqrt{e}$,
\item $a:=e/u^3$,
\item $c:=(b+\omega)/(b+\omega^2)$,
\item $G_k(X):=X^{2q^k+1}+hX+e$ for any nonnegative integer $k$,
\item $H_k(X):=X^{2q^k+1}+X+a$ for any nonnegative integer $k$,
\item $\Gamma_k$ is the set of roots in $\F_{q^3}$ of $H_k(X)$.
\end{itemize}

\begin{lemma} \label{scale}
We have $a\in\F_q\setminus\F_2$,\, $b,e,u\in\F_q\setminus\F_4$, and $b^2+b=a^{-1}+1$, and the roots of $G_\ell(X)$ are the products of $u$ with each root of $H_\ell(X)$.  
\end{lemma}

\begin{proof}
The definitions imply that $a\in\F_q^*$ and $e,b,h,u\in\F_q\setminus\F_4$. We compute
\[
(b^2+b+1)^2=\frac{1+e+e^2}{e^2}=\frac{h^3}{e^2}=\frac{1}{a^2},
\]
so that $b^2+b+1=1/a$, whence $a\ne 1$. Since $u\in\F_q^*$ we have $u^{-3} G_\ell(uX)=H_\ell(X)$, so the roots of $G_\ell(X)$ are $u$ times the roots of $H_\ell(X)$.
\end{proof}

The following result generalizes Corollary~\ref{zheng2}.

\begin{prop} \label{zhengprop2}
The polynomial $G_\ell(X)$ has either zero or three roots in $\mu_{q^2+q+1}$. 
It has three such roots if and only if $(e+\omega)^{(q^2-1)/3}=\omega^{-\ell}$, in which case these roots are the three roots of $X^3+h^2 X^2 + (e+1)hX+1$.  
Explicitly, these roots are the values $h^2+e\sqrt{h}(v+v^{-1})$ where $v$ varies over the cube roots of $\omega c$.
\end{prop}

\begin{proof}
We first show that $G_\ell(X)$ has no roots in $\mu_{q^2+q+1}\cap\F_q$. For, any such root $\beta$ would satisfy $\beta^3=1$ and $1+h\beta+e=0$, which yields the contradiction $0=h^3+e^2+e+1=(e+1)^3+e^2+e+1=e^3$.

By Lemma~\ref{scale}, the roots of $G_\ell(X)$ in $\mu_{q^2+q+1}$ are the values $u\delta$ where $\delta\in\F_{q^3}\setminus\F_q$ satisfies $H_\ell(\delta)=0$ and $\delta^{q^2+q+1}=1/u^{q^2+q+1}$.
Since $u\in\F_q$, the last condition says $\delta^{q^2+q+1}=1/u^3$, which by definition equals $a/e=ab^2$. Since $u\notin\F_4$, we know that $ab^2\ne 1$.
Thus the elements $\delta$ consist of the roots of irreducible monic cubic polynomials $g(X)\in\F_q[X]$ which divide $H_\ell(X)$ and have constant term $ab^2$.
By Lemma~\ref{scale}, the elements $a$ and $b$ satisfy the hypotheses of Proposition~\ref{fi}, so that $g(X)$ is a constant times either $f_1(X)$ or $f_2(X)$.
Plainly the ratio of the coefficients of $f_1(X)$ of degrees $0$ and $3$ is $ab^2$, while the corresponding ratio for $f_2(X)$ is $a(b+1)^2\ne ab^2$.
It follows that $G_\ell(X)$ has either zero or three roots in $\mu_{q^2+q+1}$, with three roots occurring if and only if $f_1(X)$ is irreducible in $\F_q[X]$ and divides $H_\ell(X)$, in which case the three roots are the roots of $f_1(X/u)$.
Lemma~\ref{irr} implies that $f_1(X)$ is irreducible in $\F_q[X]$ if and only if $(\omega c)^{(q^2-1)/3}\ne 1$.
Since the polynomials $f_0(X)$, $f_1(X)$, and $f_2(X)$ are pairwise coprime by Lemma~\ref{factorization}, the combination of Theorem~\ref{maincubic} and Lemma~\ref{extra} implies that $f_1(X)$ divides $H_\ell(X)$ if and only if one of the following holds:
\begin{itemize}
\item $m \equiv \ell\pmod 3$ and $(\omega^2 c)^{(q^2-1)/3}=1$;
\item $m \equiv -\ell\pmod 3$ and $c^{(q^2-1)/3}=1$;
\item $3\mid m$ and $c^{(q^2-1)/3} = \omega^{-\ell}$.
\end{itemize}
Since $(b+\omega)(b+\omega^2)=b^2+b+1=u^3/e$, we have
\[
c^{(q^2-1)/3} = \bigl((b+\omega)^2e\bigr)^{(q^2-1)/3}=(1+\omega^2e)^{(q^2-1)/3}.
\]
Since $(q^2-1)/3 \equiv m \pmod 3$, we conclude that $G_\ell(X)$ has three roots in $\mu_{q^2+q+1}$ if and only if $(e+\omega)^{(q^2-1)/3} = \omega^{-\ell}$.

We have shown that if $G_\ell(X)$ has three roots in $\mu_{q^2+q+1}$ then these three roots are the roots of
\[
u^6 f_1(X/u) = X^3 + h^2 X^2 + (e+1) h X + 1.
\]
By Lemma~\ref{f1roots}, these roots are $u(a^{-1}+v+v^{-1})/b^2$ where $v^3=\omega c$, which equals $h^2+e\sqrt{h}(v+v^{-1})$.
\end{proof}

Our final result generalizes Corollary~\ref{zheng}.

\begin{prop} \label{zhengprop}
Suppose that $m\not\equiv 1\pmod 3$. Then the set $\Gamma$ of roots in\/ $\F_{q^3}$ of $G_2(X)$ satisfies $\abs{\Gamma}=3$, and $\Gamma$ is contained in\/ $\F_q$ if and only if $1+\omega e$ is a cube in\/ $\F_{q^2}$. 
If $\Gamma\not\subseteq\F_q$ then $\Gamma\subseteq\F_{q^3}\setminus\F_q$, and $\Gamma$ is the set of roots of $F(X)$ where
\begin{enumerate}
\item $F(X)=X^3+h^2 X^2 + (e+1)hX+1$\, if\, $(e+\omega)^{(q^2-1)/3}=\omega$; 
\item $F(X)=(e+1)X^3+h^2 X^2+hX+e^2+1$\, if\, $(e+\omega)^{(q^2-1)/3}\ne\omega$.
\end{enumerate}
Explicitly, we have
\[
\Gamma=
\begin{cases}
\{\sqrt{h}(v+v^{-1})\colon v^3=c\} & \text{if $(e+\omega)^{(q^2-1)/3}=\omega^m$;} \\
\{h^2+e\sqrt{h}(v+v^{-1})\colon v^3=\omega c\} &\text{if $(e+\omega)^{(q^2-1)/3}=\omega$;} \\
\Bigl\{\displaystyle{\frac{h^2+e\sqrt{h}(v+v^{-1})}{e+1}} \colon v^3=\omega^2 c\Bigr\} &\text{otherwise.}
\end{cases}
\]
\end{prop}

\begin{proof}
By Lemma~\ref{scale} we have $a,b\in\F_q\setminus\F_2$ with $b^2+b=a^{-1}+1$ and  $\Gamma=\{u\delta\colon\delta\in\Gamma_2\}$. Thus $\Tr(1/a)=\Tr(1)$, so Theorem~\ref{main} implies that $\abs{\Gamma}=\abs{\Gamma_2}=3$.

We now determine $\Gamma_2$.
Defining $\Lambda_i$ as in Theorem~\ref{roots}, that result implies the following:
\renewcommand{\labelenumi}{(\alph{enumi})}
\begin{enumerate}
\item $c$ is a cube in $\F_{q^2}$ if and only if $\Gamma_2=\Lambda_0$;
\item if $3\mid m$ and $c$ is a non-cube in $\F_{q^2}$ then $\Gamma_2=\Lambda_k$ for the unique $k\in\{1,2\}$ such that $c^{(q^2-1)/3}=\omega^k$;
\item if $m\equiv 2\pmod 3$ and $c$ is a non-cube in $\F_{q^2}$ then there is a unique $k\in\{1,2\}$ for which $\omega^k c$ is a cube in $\F_{q^2}$, and this $k$ satisfies $\Gamma_2=\Lambda_{3-k}$.
\end{enumerate}

We now translate the above conditions on $c$ into conditions on $e$. Since $(b+\omega)(b+\omega^2)=b^2+b+1=u^3/e$, we have
\[
c^{(q^2-1)/3}=\bigl((b+\omega)^2e\bigr)^{(q^2-1)/3}=(1+\omega^2e)^{(q^2-1)/3}.
\]
Thus $c$ is a cube in $\F_{q^2}$ if and only if $1+\omega^2e$ is a cube, which is equivalent to $1+\omega e$ being a cube since $(1+\omega e)(1+\omega^2 e)=h^3$ is a cube. 
If $3\mid m$ then $\omega^{(q^2-1)/3}=1$, so that $c^{(q^2-1)/3}=(\omega+e)^{(q^2-1)/3}$.  
If $m\equiv 2\pmod 3$ then $\omega^{(q^2-1)/3}=\omega^2$, so that $c^{(q^2-1)/3}=\omega(\omega+e)^{(q^2-1)/3}$, and also $\omega^k c$ is a cube in $\F_{q^2}$ if and only if $c^{(q^2-1)/3}=\omega^k$.

We have shown that $1+\omega e$ is a cube in $\F_{q^2}$ if and only if $\Gamma_2=\Lambda_0$. 
Moreover, if $1+\omega e$ is not a cube in $\F_{q^2}$ then $\Gamma_2=\Lambda_1$ if and only if $(\omega+e)^{(q^2-1)/3}=\omega$, and $\Gamma_2=\Lambda_2$ otherwise. 
By Lemmas~\ref{f0roots} and \ref{f1roots}, for each $i\in\{0,1,2\}$ the set $\Lambda_i$ is the set of roots of the polynomial $f_i(X)$ from Section~\ref{sec:notation}. 
Since plainly every element of $\Gamma_2\cap\F_q$ is a root of $f_0(X)$, Proposition~\ref{fi} implies that $\Gamma_2=\Lambda_0$ if and only if $\Gamma_2\subseteq\F_q$. 
Now the result follows, since the polynomials $F(X)$ in items (1) and (2) of Proposition~\ref{zhengprop} are $h^3 f_1(X/u)$ and $h^3 f_2(X/u)$, respectively, and the three cases in the description of $\Gamma$ in Proposition~\ref{zhengprop} are $\{u\delta\colon \delta\in \Lambda_i\}$ for $i=0,1,2$ in that order.
\end{proof}



\end{document}